\documentclass[a4paper]{amsart}
\usepackage{ifpdf}
\usepackage[colorlinks,bookmarksopen]{hyperref}
\usepackage[all]{xypic}
\usepackage{amssymb}

\title{Cohomology of products and coproducts of augmented algebras}
\author{Matthew Towers}

\newcommand{\ext}{\operatorname{Ext}}
\newcommand{\tor}{\operatorname{Tor}}

\newcommand{\HH}{\operatorname{HH}}
\newcommand{\iHH}{\operatorname{iHH}}
\newcommand{\Hom}{\operatorname{Hom}}

\newcommand{\im}{\operatorname{im}}

\newcommand{\Aa}{\mathcal{A}}
\newcommand{\eps}{\varepsilon}
\newcommand{\id}{\operatorname{id}}

\newtheorem{theo}{Theorem}[section]
\newtheorem{lem}[theo]{Lemma}

\newtheorem{prop}[theo]{Proposition}
\newtheorem{rem}[theo]{Remark}

\newtheorem{ex}[theo]{Example}

\begin{document}
\begin{abstract} We show that the ordinary cohomology functor $\Lambda \mapsto \ext ^* _\Lambda (k,k)$ from the category of augmented $k$-algebras to itself exchanges coproducts and products, then that Hochschild cohomology is close to sending coproducts to products if the factors are self-injective.  We identify the multiplicative structure of the Hochschild cohomology of a product modulo a certain ideal in terms of the cohomology of the factors.
\end{abstract}
\maketitle
\begin{section}{Introduction}
An augmented $k$-algebra is a $k$-algebra $\Lambda$ equipped with a homomorphism $\Lambda \to k$.  The aim of this paper is to study how the ordinary cohomology $\ext_\Lambda ^* (k,k)$ and the
Hochschild cohomology $\HH(\Lambda)$ of such algebras interact with products and coproducts.

We begin with a brief overview of the category $\mathcal{A}$ of augmented $k$-algebras in Section \ref{cat_section}, including its product and coproduct.  We collect some useful information on Hochschild cohomology in Section \ref{hoch_section}.  In Section \ref{coproduct_section} we look at cohomology of coproducts, showing in Theorem \ref{ordinary_coprod} that the contravariant functor $E$ from $\mathcal{A}$ to itself sending $\Lambda$ to $\ext_\Lambda ^*(k,k)$  maps coproducts to products, and in Theorem \ref{hoch_coprod} that Hochschild cohomology very nearly does the same thing if the factors are self-injective.  These results are obtained easily using dimension shifting. Finally in Section \ref{product_section} we look at the cohomology of products of augmented algebras.  Given two augmented algebras $\Lambda$ and $\Gamma$ we construct a `small' bimodule resolution of the product $\Lambda * \Gamma$ from small bimodule resolutions of the factors.  Here a resolution $P$ being small means that the image of the differential is contained in  $I \cdot P$ where $I$ is the augmentation ideal --- if $\Lambda$ is a finite dimensional local algebra then this is equivalent to the resolution being minimal in the usual sense.  Our resolution can be used to study ordinary cohomology, and we use it in Theorem \ref{main_theo} to prove that the functor $E$ given above sends products to coproducts.  We then use a long exact sequence to investigate the Hochschild cohomology of $\Lambda * \Gamma$.  This is a long way from being the coproduct of the Hochschild cohomology algebras of its factors: there is no hope of this since Hochschild cohomology is graded commutative.  In Theorem \ref{hoch_prod}  we show that, modulo a certain ideal, $\HH(\Lambda * \Gamma)$ is the direct sum of the product of certain subalgebras of $\HH(\Lambda)$ and $\HH(\Gamma)$ with terms related to $E(\Lambda)$ and $E(\Gamma)$.  We also show in Proposition \ref{nilp_HH} that if $I(\Lambda)$ and $I(\Gamma)$ are nilpotent then every element of positive degree in the Hochschild cohomology of the product $\Lambda * \Gamma$ is  nilpotent with nilpotency degree at most the maximum of the nilpotency degrees of the augmentation ideals of the two factors.

Versions of our results on ordinary cohomology have already appeared in the literature, for example it was proved for graded algebras $A$ with $A_0=k$ in \cite[Proposition 1.1]{PP}, for local Noetherian commutative algebras over a field in \cite{Moore},  and for coproducts of groups in \cite[VI.14]{HS} (though the latter gives only an isomorphism of vector spaces).  

These results on $E$, regarded as a covariant functor $\mathcal{A} \to \mathcal{A}^{\mathrm{op}}$, amount to saying that it preserves certain (co)limits.  If $E$ had a left- (right-) adjoint this would be automatic \cite[V.5]{MacL}.  We show that neither adjoint exists in Section \ref{cat_section}.

\end{section}


\begin{section}{The category of augmented algebras} \label{cat_section}

Let $k$ be a field.  We define the category $\Aa$ of augmented $k$-algebras as follows: an object of $\Aa$ is a pair $(\Lambda, \varepsilon_\Lambda)$ where $\Lambda$ is a unital $k$-algebra and $\varepsilon _\Lambda $ is an algebra homomorphism $ \Lambda \to k$.  
We write $I(\Lambda)$ for $\ker \varepsilon_\Lambda$.  
A morphism in $\Aa$ between $(\Lambda,\varepsilon_\Lambda)$ and $(\Gamma,\varepsilon_\Gamma)$ is a $k$-algebra homomorphism $f:\Lambda \to \Gamma$ that preserves the augmentations, that is $\varepsilon_\Lambda = \varepsilon_\Gamma \circ f$ (so that $\Aa$ is a comma category in the sense of \cite[II.6]{MacL}).  An augmented algebra is an algebra that appears as the first coordinate of an object of $\Aa$.

\begin{ex} 
Any group algebra $kG$ is augmented with augmentation given by $g \mapsto 1$ for all $g \in G$.  If $Q$ is a quiver and $I$ is an admissible ideal of $kQ$ then for each vertex $e$ of $Q$ there is an augmentation $\eps _e$ of $kQ/I$ sending (the image of) each arrow to zero, $e$ to $1$, and every other vertex to zero.
\end{ex}

Given any two augmented algebras $\Lambda$ and $\Gamma$ we have a diagram
\[ \xymatrix{ & \Lambda \ar[d]^{\varepsilon_\Lambda} \\
\Gamma \ar[r]^{\varepsilon_\Gamma} & k } \]
in the category of $k$-algebras.  We define $\Lambda * \Gamma$ to be the  pullback of this diagram, so that $\Lambda * \Gamma= \{ (\lambda, \gamma) \in \Lambda \oplus \Gamma : \eps_\Lambda \lambda = \eps_\Gamma \gamma \}$.  Let $p_\Lambda$ and $p_\Gamma$ be the projection maps from the pullback to $\Lambda$ and $\Gamma$.  Notice $\Lambda * \Gamma$  is an augmented algebra when equipped with the augmentation $\eps_{\Lambda *\Gamma} = \eps_\Lambda \circ p_\Lambda = \eps _\Gamma \circ p_\Gamma$.  We have:

\begin{lem}
$* $ is a product in $\mathcal{A}$ with canonical projections $p_\Lambda$ and $p_\Gamma$.
\end{lem}
%
For example, the product of $k[x]/x^2$ with itself is isomorphic to $k[x,y]/(x^2, y^2, xy)$.  These products can occur naturally:

\begin{ex} 
Let $k$ have characteristic $3$ and let \[D =\langle x,y,z | x^3, y^3, z^3, [x,y]=z \in Z(D) \rangle\] be the extra-special group of order $27$ and exponent $3$.  Then the endomorphism algebra of the transitive permutation module $k _{\langle x\rangle}\!\! \uparrow ^D$ is
\[  k[a]/a^3 * k[b]/b^2 * k[c]/c^2. \]
For further examples of this type see \cite{mt}.
\end{ex}

%
%

This product satisfies the following Chinese Remainder-type result:
\begin{lem}
 Let $\Lambda$ be an augmented $k$-algebra with ideals $I,J$ such that $I+J=I(\Lambda)$ and $I\cap J = IJ$.  Then $\Lambda/IJ \cong \Lambda/I * \Lambda/J$.
\end{lem}

In $\Aa$ there is a free product construction like that in the category of groups.  Given augmented $k$-algebras $\Lambda$ and $\Gamma$ let $R_0=k$ and $R_n = (I(\Lambda) \otimes _k I(\Gamma) \otimes_k \ldots)  \oplus  (I(\Gamma) \otimes_k I(\Lambda) \otimes_k \ldots)$ where both summands have length $n>0$.  Let $\Lambda \sqcup \Gamma$ be $\bigoplus_{n\geq 0} R_n$ equipped with multiplication 
\[ (u_1 \otimes \cdots \otimes u_k ) \cdot ( v_1 \otimes \cdots \otimes v_l ) = u_1 \otimes \cdots \otimes (u_k v_1) \otimes v_2 \otimes \cdots \otimes v_l \]
if $u_k$ and $v_1$ are both in $\Lambda$ or both in $\Gamma$, and 
\[u_1 \otimes \cdots \otimes u_k \otimes v_1 \otimes v_2 \otimes \cdots \otimes v_l\]
otherwise.  This construction makes $\Lambda \sqcup \Gamma$ into an augmented algebra, the augmentation quotienting out all $R_n$ with $n>0$.  Write $i_\Lambda$ and $i_\Gamma$ for the obvious inclusions in degrees zero and one.  We then have:

\begin{lem} $\sqcup $ is a coproduct in $\mathcal{A}$ with canonical injections $i_\Lambda$ and $i_\Gamma$. \end{lem}


For example, $k[x] \sqcup k[y]$ is the free associative algebra on two variables.

\begin{rem}
Define a local $k$-algebra to be a $k$-algebra $\Lambda$ with a unique maximal left ideal $\mathfrak{m}$ such that $\Lambda / \mathfrak{m} \cong k$ (any such ideal is automatically two-sided).  Let $\mathcal{L}$ be the category of local algebras, where the morphisms are homomorphisms of algebras.  Then the obvious inclusion is a fully faithful embedding $\mathcal{L} \hookrightarrow \Aa$ (fullness follows because any map of local algebras automatically preserves the maximal ideals).  Then $*$ is a product in $\mathcal{L}$, but $\sqcup$ isn't a coproduct in $\mathcal{L}$.  The coproduct in $\Aa$ of $k[x]/x^2$ and $k[y]/y^2$  is $k\langle x,y \rangle / (x^2, y^2)$, which is not local as neither $xy$ nor $1-xy$ has a left-inverse (in a local algebra the non-units form a left-ideal \cite[5.21]{CR1}).  We conjecture that in general $\mathcal{L}$ does not have a coproduct of $k[x]/x^2$ and $k[y]/y^2$.
\end{rem}

The one-dimensional algebra $k$ is both an initial object (empty product) and terminal object (empty coproduct) in $\Aa$.  
%


Write $E(\Lambda)$ for the ordinary cohomology ring $\ext^*_\Lambda(k,k)$, augmented by the map killing all elements of positive degree, and write $E^n(\Lambda)$ for $\ext^n_\Lambda(k,k)$.  Change of rings \cite[VIII, \S3]{CE} can be used to make $E$ into a contravariant endofunctor on $\Aa$.  Let $f$ be an augmentation preserving map $\Lambda \to \Gamma$, so that $E(f)$ should be a map $E(\Gamma) \to E(\Lambda)$.  The definition of $E(f)$ is particularly straightforward when we regard $\ext^n_\Gamma(k,k)$ as the set of equivalence classes of exact sequences of $\Gamma$ modules of the form
\[ 0 \to k \to N_{n-1} \to \cdots \to N_0 \to k \to 0. \]
$E(f)$ simply maps the class of this sequence to the class of the same sequence, regarded as a sequence of $\Lambda$ modules via $f$. 
More commonly we will define $E(f)$ as follows: take projective resolutions $
P_*$ and $Q_*$ of $k$ as a $\Lambda$ and $\Gamma$ module respectively.  The $Q_n$ become $\Lambda$ modules via $f$, so there is a chain map $H$ lifting the identity $k\to k$:
\[
\xymatrix{
P_* \ar[r] \ar[d]^H & k \ar@{=}[d] \\
Q_* \ar[r] &k
} 
\]
Now to define $E(f)$ on an element $a$ of $\ext^n_\Gamma (k,k)$ we pick a representative cocycle $\alpha : Q_n \to k$ and let $E(f)(a)$ be the cohomology class of $\alpha \circ H: P_n \to k$. 

Some of the results we prove in the next two sections amount to saying that the functor $E$ preserves certain (co)limits.  This would be immediate if a left (right) adjoint existed.

\begin{lem}
$E$ has no right or left adjoint.
\end{lem}
\begin{proof}
Recall that $E(k[x]/x^2)=k[a]$ with $a$ in degree one, and for $r>2$, $E(k[x]/x^r)=k[a,b]/a^2$ with $a$ in degree one and $b$ in degree two.  

If a contravariant functor has a left adjoint, it sends monomorphisms to epimorphisms.  Apply $E$ to the monomorphism $f: k[x]/x^2 \to k[x]/x^4$ that sends $x$ to $x^2$.  The resulting map $E(f): k[a,b]/a^2 \to k[a]$ is zero in degree one, and is not an epimorphism (composing with the zero and quotient maps $k[a] \to k[a]/a^2$ gives the same result).

If $E$ had a right adjoint it would have to send epimorphisms to monomorphisms.  This time let $f$ be the epimorphism $k[x]/x^3 \to k[x]/x^2$ given by  $x\mapsto x$.  Then $E(f)$ is zero in degree one, therefore not a  monomorphism (the monomorphisms in $\Aa$ are exactly the one-to-one homomorphisms, although epimorphisms need not be onto).
\end{proof}

The same examples can be adapted to show that $E$ does not exchange equalisers and coequalisers.
 
\end{section}


\begin{section}{Hochschild cohomology} \label{hoch_section}

In this section we review some standard definitions and results about Hochschild cohomology.

Given an algebra $\Lambda$, we define the Hochschild cohomology ring $\HH(\Lambda)$ to be $\bigoplus _{n\geq0} \ext ^n _{\Lambda ^e} (\Lambda, \Lambda)$ equipped with the Yoneda product, and we write $\HH^n(\Lambda)$ for $\ext ^n _{\Lambda ^e} (\Lambda, \Lambda)$. Here $\Lambda ^e$ is the enveloping algebra $\Lambda \otimes _k \Lambda ^{\mathrm{op}}$ and $\Lambda$ is regarded as a left $\Lambda ^e$ module by $(\lambda \otimes \gamma) \cdot x = \lambda x \gamma$.  The categories of left $\Lambda^e$ modules and $\Lambda$-$\Lambda$ bimodules are isomorphic, and we will often work with bimodules for convenience. $\HH(\Lambda)$ is a graded-commutative
%
%
$k$-algebra under the Yoneda product, and $\HH^0(\Lambda)$ is the centre $Z(\Lambda)$.

Hochschild cohomology is not functorial in $\Lambda$.  Functoriality can be recovered by working instead with $\ext ^* _{\Lambda ^e} (\Lambda, \Hom_k(\Lambda,k))$ (this is done for example in \cite{Loday, bensonII}),
the disadvantage of doing this is that there is no obvious ring structure. The two definitions coincide only for symmetric algebras.

If $\Lambda$ is augmented then there is an exact sequence of $\Lambda$-$\Lambda$ bimodules
\[ \label{a_ses} 0 \to I(\Lambda) \to \Lambda \stackrel{\epsilon} {\to} k \to 0. \]
Applying $\Hom_{\Lambda^e} (\Lambda, -)$ gives a long exact sequence
\begin{equation} \label{ILles}
\cdots \to \ext ^n _{\Lambda^e} (\Lambda, I(\Lambda)) \to \HH^n(\Lambda) \stackrel{\phi_k}{\to} E^n(\Lambda) \stackrel{\omega_\Lambda}{\to} \ext ^{n+1}_{\Lambda^e}(\Lambda, I(\Lambda)) \to \cdots
\end{equation}
where we have made the identification $\ext_{\Lambda^e} ^n(\Lambda, k ) \cong E^n(\Lambda)$ using \cite[X, Theorem 2.1]{CE}.
If $\Delta$ is a graded ring then the \textbf{graded centre} of $\Delta$ is defined to be  the span of all homogeneous elements $z$ of $\Delta$ such that $zg=(-1)^{\operatorname{dg}(z)\operatorname{dg}(g)}gz$ for all homogeneous $g \in \Delta$.  The map $\phi_k$ is a homomorphism of graded rings, and \cite[Theorem 1.1]{SS} shows its image lies in the graded centre of $E(\Lambda)$  (in their notation, take $M=N=\Lambda$ and $\Gamma=k$).   If $I(\Lambda)^n=0$, then the kernel of $\phi_k$ is nilpotent with each element having nilpotency index at most $n$ by \cite[Proposition 4.4]{SS} --- the result stated there is for the radical rather than the augmentation ideal, but the proof goes through without changes.
\end{section}


\begin{section}{Cohomology of coproducts} \label{coproduct_section}

We begin with ordinary cohomology:

\begin{theo} \label{ordinary_coprod}
Let $\Lambda$ and $\Gamma$ be augmented $k$-algebras.
%
%
Then
\[ \ext ^* _{\Lambda \sqcup \Gamma} (k,k) \cong \ext^* _\Lambda (k,k) * \ext ^*_\Gamma (k,k) \]
\end{theo}

\begin{proof}
 Notice that 
\begin{equation} \label{omega1} I(\Lambda \sqcup \Gamma) \cong I(\Lambda) \!\!\uparrow ^{\Lambda \sqcup \Gamma} \oplus I(\Gamma ) \!\!\uparrow ^{\Lambda \sqcup \Gamma} \end{equation}
 and also that $\Lambda \sqcup \Gamma$ is free on restriction to $\Lambda$ and $\Gamma$, regarded as subalgebras via the canonical injections $i_\Lambda$ and $i_\Gamma$. We can now compute $E(\Lambda \sqcup \Gamma)$ by dimension shifting:
\begin{align*}
 \ext^n_{\Lambda \sqcup \Gamma}(k,k) &\cong \ext ^{n-1}_{\Lambda \sqcup \Gamma} (I(\Lambda \sqcup \Gamma), k) \\
 &\cong \ext ^{n-1} _{\Lambda \sqcup \Gamma} (I(\Lambda) \!\!\uparrow^{\Lambda \sqcup \Gamma}, k) \oplus \ext ^{n-1} _{\Lambda \sqcup \Gamma} (I(\Gamma) \!\!\uparrow^{\Lambda \sqcup \Gamma}, k) \\
&\cong \ext ^{n-1}_{\Lambda}(I(\Lambda), k) \oplus \ext ^{n-1}_{\Gamma}(I(\Gamma), k) \\
&\cong \ext ^{n}_{\Lambda}(k,k) \oplus \ext ^{n}_{\Gamma}(k,k)
\end{align*}
for $n>0$, so that $\ext ^n _{\Lambda \sqcup \Gamma} (k,k)$ and $\ext^n _\Lambda (k,k) * \ext ^n_\Gamma (k,k)$ agree as vector spaces for all $n\geq 0$.
Take projective resolutions $P_*$ and $Q_*$ of $k$ as $\Lambda$ and $\Gamma$ modules respectively.  Because of (\ref{omega1}), we know there is a projective resolution of $k$ as a $\Lambda\sqcup\Gamma$ module which is $P_*\uparrow^{\Lambda\sqcup\Gamma} \oplus Q_*\uparrow^{\Lambda\sqcup\Gamma}$ in degrees greater than zero.  It follows that the linear isomorphism $E(\Lambda\sqcup\Gamma) \to E(\Lambda)*E(\Gamma)$  can be realised as $E(i_\Lambda)* E(i_\Gamma)$, the map arising from the universal property of the product,  and is therefore an homomorphism of algebras.
\end{proof}

For Hochschild cohomology we can procede similarly, but the version of (\ref{omega1}) required is slightly harder.  Given an augmented algebra $A$ write $\Omega _{A^e}$ for the kernel of the bimodule map $d_0: A^e \to A$ defined by $a_1 \otimes a_2 \mapsto a_1a_2$.  

\begin{lem} \label{omega_lem}  $\Omega  _{(\Lambda \sqcup \Gamma)^e}  \cong \Omega _{\Lambda ^e}  \!\!\uparrow ^{(\Lambda \sqcup \Gamma)^e} \oplus \Omega _{\Gamma ^e}  \!\!\uparrow ^{(\Lambda \sqcup \Gamma)^e}$ \end{lem}

\begin{proof} 
If $A$ is an augmented algebra then $\Omega _{A^e}$ is generated as an $A$-$A$ bimodule by elements of the form $a \otimes 1 - 1 \otimes a$ for $a \in A$, by \cite[IX, Proposition 3.1]{CE}.  In fact if $S$ is a generating set for $A$ as an algebra then the elements $s \otimes 1 - 1 \otimes s$ for $s \in S$ generate $\Omega _{A^e}$ as a bimodule.  Now $\Omega _{(\Lambda \sqcup \Gamma)^e}$ contains copies $O_\Lambda$ and $O_\Gamma$ of $\Omega_{\Lambda^e}\!\!\uparrow ^{(\Lambda \sqcup \Gamma)^e}$ and $\Omega_{\Gamma^e}\!\!\uparrow ^{(\Lambda \sqcup \Gamma)^e}$ consisting of the submodules  generated by elements of the form 
\[ 1\otimes \lambda - \lambda \otimes 1 \,\,\,\mathrm{and}\,\,\, 1 \otimes \gamma - \gamma \otimes 1
 \]
respectively, where $\lambda \in I(\Lambda)$ and $\gamma \in I(\Gamma)$.  Between them these generate $\Omega _{(\Lambda \sqcup \Gamma)^e}$, so our task is to show that they intersect trivially. 

Fix bases $\{ \lambda_i : i \in I \}$ of $I(\Lambda)$ and $\{ \gamma_j : j \in J \}$ of $I(\Gamma)$.  Then $\Omega _{(\Lambda \sqcup \Gamma)^e}$ is linearly spanned by elements of the form
\begin{eqnarray}
 \label{1} \cdots \lambda_i \otimes \lambda_j \gamma_k \cdots & - & \cdots \lambda_i \lambda_j \otimes \gamma_k \cdots \\ \label{2}
 \cdots \lambda_i \otimes \gamma_j \lambda_k \cdots & - & \cdots \lambda_i \gamma_j \otimes \lambda_k \cdots \\ \label{3}
\cdots \gamma_i \otimes \lambda_j \gamma_k \cdots & - & \cdots \gamma_i \lambda_j \otimes \gamma_k \cdots \\\label{4}
\cdots \gamma_i \otimes \gamma_j \lambda_k \cdots & - & \cdots \gamma_i \gamma_j \otimes \lambda_k \cdots \\\label{5}
\cdots \gamma_i \lambda_j \otimes \lambda_k \cdots &-& \cdots\gamma_i \otimes \lambda_j \lambda_k \cdots \\\label{6}
\cdots \lambda_i \gamma_j \otimes \gamma_k \cdots &-& \cdots\lambda_i \otimes \gamma_j \gamma_k \cdots \\\label{7}
1 \otimes \lambda_i \gamma_j \cdots&-& \lambda_i \otimes \gamma_j \cdots \\\label{8}
1 \otimes \gamma_i \lambda_j \cdots &-& \gamma_i\otimes  \lambda_j \cdots \\\label{9}
1 \otimes \lambda_i \lambda_i \cdots &-& \lambda_i \otimes \lambda_j \cdots \\\label{10}
1 \otimes \gamma_i \gamma_i \cdots &-& \gamma_i \otimes \gamma_j \cdots \\\label{11}
\cdots \gamma_i \otimes \lambda_j &-& \cdots \gamma_i \lambda_j \otimes 1 \\\label{12}
\cdots \lambda_i \otimes \gamma_j &-& \cdots \lambda_i \gamma_j \otimes 1 \\\label{13}
\cdots \lambda_i \otimes \lambda_j &-& \cdots \lambda_i \lambda_j \otimes 1 \\ \label{14}
\cdots \gamma_i \otimes \gamma_j &-& \cdots \gamma_i \gamma_j \otimes 1 
\end{eqnarray}
$O_\Lambda$ is spanned by elements of types (\ref{1}),  (\ref{3}),  (\ref{5}),  (\ref{7}),  (\ref{9}),  (\ref{11}),  and (\ref{13}), while $O_\Gamma$ is spanned by elements of types (\ref{2}),  (\ref{4}),  (\ref{6}),  (\ref{8}),  (\ref{10}),  (\ref{12}),  and (\ref{14}).  Many of these terms cannot appear with non-zero coefficient in $W$: for example no term of the form (\ref{1}) can appear, because $O_\Gamma$ contains no tensor of the form $ \cdots \lambda_i \otimes \lambda_j  \cdots$.  This type of argument shows that the only way $O_\Lambda \cap O_\Gamma$ can be non-zero is if there is a non-trivial linear dependence between elements of types (\ref{2}) and (\ref{3}).  Suppose such a linear dependence of the form $W=V$ exists, where $0 \neq W \in O_\Lambda$ and $V \in O_\Gamma$.  Suppose 
$\cdots \lambda_l \gamma_i \otimes \lambda_j \gamma_k \cdots  -  \cdots \lambda_l \gamma_i \lambda_j \otimes \gamma_k \cdots$ appears with non-zero coefficient in $W$, and is chosen so that the first term $\cdots \lambda_l \gamma_i \otimes \lambda_j \gamma_k \cdots$ has its left hand factor $\cdots \lambda_l \gamma_i$ of minimal length, where the length of a word $w_1...w_n$ alternating between $\lambda_i$s and $\gamma_j$s is $n$. Now $\cdots \lambda_l \gamma_i \otimes \lambda_j \gamma_k \cdots$ must also appear in $V$ --- the only way it can do so is as an element
\[\cdots \lambda_l \gamma_i \otimes \lambda_j \gamma_k \cdots - \cdots \lambda_l \otimes \gamma_i \lambda_j \gamma_k \cdots\] 
appears in $V$ with non-zero coefficient.  Now the second term must appear in $W$, but its left-hand factor is shorter than $\cdots \lambda_l \gamma_i \otimes \lambda_j \gamma_k \cdots$.  This is a contradiction.
\end{proof}

Now we can compute $\HH(\Lambda \sqcup \Gamma)$ by dimension-shifting.  If $n>1$
\begin{align}
\nonumber \ext^n_{(\Lambda \sqcup \Gamma)^e}(\Lambda \sqcup \Gamma,\Lambda \sqcup \Gamma) &= \ext ^{n-1}_{(\Lambda \sqcup \Gamma)^e} (\Omega_{( \Lambda \sqcup \Gamma)^e}, \Lambda \sqcup \Gamma) \\ \nonumber
&= \ext ^{n-1} _{(\Lambda \sqcup \Gamma)^e} (\Omega_{\Lambda^e} \!\!\uparrow^{(\Lambda \sqcup \Gamma)^e}, \Lambda \sqcup \Gamma) \oplus \ext ^{n-1} _{(\Lambda \sqcup \Gamma)^e} (\Omega_{\Gamma^e} \!\!\uparrow^{(\Lambda \sqcup \Gamma)^e}, \Lambda \sqcup \Gamma) \\ \nonumber
&= \ext ^{n-1}_{\Lambda^e}(\Omega_{\Lambda^e} , {\Lambda \sqcup \Gamma}) \oplus \ext ^{n-1}_{\Gamma^e}(\Omega_{\Gamma^e} , {\Lambda \sqcup \Gamma}) \\\label{lastline}
&= \ext ^{n}_{\Lambda^e}(\Lambda, {\Lambda \sqcup \Gamma}) \oplus \ext ^{n}_{\Gamma^e}(\Gamma, {\Lambda \sqcup \Gamma})
\end{align}
using that $(\Lambda \sqcup \Gamma)^e$ is free on restriction to $\Lambda^e$ and $\Gamma^e$. We can now describe the structure of $\HH(\Lambda \sqcup \Gamma)$. 

\begin{theo} \label{hoch_coprod}
If $\Lambda^e$ and $\Gamma^e$ are self-injective then
\[ \HH^n(\Lambda \sqcup \Gamma) \cong \HH^n(\Lambda) \oplus \HH^n(\Gamma) \]
for $n>1$, and there is a direct sum decomposition \[\HH^1(\Lambda \sqcup \Gamma) \cong \HH^1(\Lambda) \oplus \HH^1 (\Gamma) \oplus K. \]  If $z_1, z_2 \in Z(\Lambda \sqcup \Gamma) \cap I(\Lambda \sqcup \Gamma)$, $k_1, k_2 \in K$ and $l_1, g_1$ and $l_2,g_2$ lie in the copies of $\HH^{\geq 1} (\Lambda)$ and $\HH^{\geq 1} (\Gamma)$ embedded in $\HH (\Lambda \sqcup \Gamma)$, then 
\[ (z_1 + k_1 +l_1+g_1) ( z_2 + k_2 +l_2+g_2)=(z_1z_2 + z_1 k_2 +z_2k_1 + l_1l_2 + g_1g_2) \]
where the multiplications $l_1l_2$ and $g_1g_2$ are carried out in $\HH(\Lambda)$ and $\HH(\Gamma)$.
\end{theo}

\begin{proof}
$\Lambda \sqcup \Gamma$ restricted to $\Lambda^e$ is $\Lambda \oplus F$ where $F$ is free, hence injective under our hypotheses, so the first statement follows from (\ref{lastline}).
If $n=1$ the dimension-shifting formula is slightly different: 
 \[ \ext^1_{(\Lambda \sqcup \Gamma)^e}(\Lambda \sqcup \Gamma,\Lambda \sqcup \Gamma) \cong \frac{\Hom_{(\Lambda \sqcup \Gamma)^e} (\Omega  _{(\Lambda \sqcup \Gamma)^e} , \Lambda \sqcup \Gamma)}{F}\]
where $F$ is the space of homomorphisms factoring through $\Omega_{(\Lambda \sqcup \Gamma)^e}  \hookrightarrow (\Lambda \sqcup \Gamma)^e$.  This is spanned by  maps  $f_w: \Omega  _{(\Lambda \sqcup \Gamma)^e} \to \Lambda \sqcup \Gamma$ for $w \in \Lambda \sqcup \Gamma$ given by 
\[f_w (x \otimes 1 - 1 \otimes x ) = x w - w x\] for $x \in I(\Lambda)\cup I(\Gamma)$. Using Lemma \ref{omega_lem} and Frobenius reciprocity we  have 
\[ \HH^1(\Lambda \sqcup \Gamma) \cong \frac{\Hom_{\Lambda ^e}(\Omega_{\Lambda^e},\Lambda) \oplus \Hom_{\Gamma ^e}(\Omega_{\Gamma^e},\Gamma)\oplus \Hom_{\Lambda^e}(\Omega_{\Lambda^e}, X)\oplus \Hom_{\Gamma^e}(\Omega_{\Gamma^e}, Y)}{F}\]
where $X$ and $Y$ are free $\Lambda^e$ and $\Gamma^e$ modules respectively.  Mapping the last two summands to zero gives a surjection of $k$-vector spaces $\HH^1(\Lambda \sqcup \Gamma) \twoheadrightarrow \HH^1(\Lambda) \oplus \HH^1(\Gamma)$, and we take $K$ to be the kernel of this map.
The formula for products in $\HH(\Lambda \sqcup \Gamma)$ follows from the existence of a projective resolution analogous to the one in the proof of Theorem \ref{ordinary_coprod}.
\end{proof}

\begin{ex}
Let $\Lambda=k[x]/x^2$ and $\Gamma=k[y]/y^2$, where $k$ has characteristic not two.  $\HH(\Lambda)$ is the graded commutative ring $k[x_0, x_1, x_2]/(x_0^2, x_1^2, x_0 x_1, x_0 x_2)$, where degrees are given by subscripts. 
We have $Z(\Lambda * \Gamma)=k[xy+yx]$, and by Theorem \ref{hoch_coprod} the algebra structure of $\HH(\Lambda * \Gamma)$ is given by 
\[ \frac{\HH(\Lambda)}{(x_0)}  * \frac{\HH(\Gamma)}{(y_0)} \oplus Z(\Lambda * \Gamma) \cap I(\Lambda * \Gamma) \oplus K \]
where the augmentations on $\HH(\Lambda)/(x_0)$ and $\HH(\Gamma)/(y_0)$ kill elements of positive degree, and $K$ is the quotient of $(\Lambda * \Gamma) \oplus (\Lambda * \Gamma)$ given by
\[
\frac{\langle 
(xyx,0), (xyxyx,0), (xyxyxyx,0)\ldots \rangle +  \langle 
 (0,yxy), (0,yxyxy),(0,yxyxyxy)\ldots \rangle}{\langle  (xyx\cdots yx, -yxy\cdots xy) \rangle} 
\]
regarded as a subspace of $\HH^1(\Lambda * \Gamma)$.
%
%
\end{ex}

\end{section}


\begin{section}{Cohomology of products} \label{product_section}

In this section we proceed as follows: in \ref{bimod_res_subsection} we use bimodule resolutions for $\Lambda$ and $\Gamma$  to build a bimodule resolution of $\Lambda * \Gamma$.  In \ref{ord_cohom_prod} we
use this resolution to study the ordinary cohomology $E(\Lambda * \Gamma)$ by applying the functor $- \otimes _{\Lambda * \Gamma} k$, and prove in Theorem \ref{main_theo} that $E(\Lambda * \Gamma) = E(\Lambda) \sqcup E(\Gamma)$, so long as $\Lambda$ and $\Gamma$ are either finite-dimensional or finitely generated and \textbf{graded connected}, that is, graded with zero degree part equal to $k$. Finally in \ref{hoch_cohom_prod} we return to our bimodule resolution to investigate Hochschild cohomology of products, showing how $\HH(\Lambda * \Gamma)$ decomposes, modulo a certain ideal, into a direct sum of terms coming from the cohomology of $\Lambda$ and $\Gamma$.

The reason for the restriction to finite-dimensional or finitely generated graded connected algebras is that we require the existence of a bimodule resolution $(P_*,d)$ of $\Lambda$ with the property that $\im d \subseteq I(\Lambda) \cdot P_* + P_* \cdot I(\Lambda)$ (and a similar resolution for $\Gamma$).  Such resolutions exist given our restriction on $\Lambda$ and $\Gamma$: if they are finite dimensional then a minimal projective resolution exists which certainly has this property. On the other hand if $\Lambda$ is graded with $\Lambda_0=k$ and and $M$ is a finitely generated graded $\Lambda$ module with $\{m_1\ldots m_r \}$ a minimal homogeneous generating set then the kernel of
the map $\bigoplus \Lambda e_i \to M$ given by $e_i \mapsto m_i$ is concentrated in positive degree.  To see this, let $\sum \lambda_i m_i=0$ and suppose that $\lambda_1$ has a non-zero degree zero component, which we may take to be $1$. Let the degree of $m_1$ be $N$ and project onto the degree $N$ component of $\sum \lambda_i m_i=0$. We get $m_1 + \sum _{i>1} [\lambda_i]_{N-\operatorname{deg}(m_i)} m_i  =0$ where $[\lambda]_m$ denotes the degree $m$ component of $\lambda\in\Lambda$.  It follows that $m_1$ was not essential in the generating set.  Using this repeatedly we can build a projective resolution with the required property. In general however, no such resolution exists as the following example shows.

\begin{ex}
Let $\Lambda = k[x,y]/(x^2+y^2-x)$, augmented by $\eps(x)=\eps(y)=0$. Then $I(\Lambda) = (x,y)$ is a principal ideal if and only if $k$ contains a square root $i$ of $-1$, in which case $I(\Lambda)=(y+ix)$ and there is a projective resolution
\[
0 \to \Lambda \stackrel{y+ix}{\to} \Lambda \stackrel \eps \to k \to 0. 
\]
If $k$ has no square root of $-1$, any projective resolution 
\[
\cdots \to Q \stackrel {d_1} \to \Lambda \to k \to 0 
\]
has at least two summands at degree one ($\Lambda$ is indecomposable as a left module over itself as it is an integral domain, so the only finitely generated projectives are free).  We claim there is no projective resolution in which $\ker d_1 \subseteq I(\Lambda) \cdot Q$.  If there were, $\ext ^1_\Lambda (k,k)$ would be isomorphic to $\Hom_\Lambda (Q, k)$ which would have dimension strictly greater than one, whereas $\ext ^1_\Lambda (k,k) \cong \Hom _\Lambda (I(\Lambda), k)$ which has dimension one.
It now follows there is no bimodule resolution $(P_*,d)$ of $\Lambda$ with the property that $\im d \subseteq I(\Lambda) \cdot P_* + P_* \cdot I(\Lambda)$, for if there were we could obtain a projective resolution $(Q_*, \delta)$ of $k$ in which $\ker \delta_1 \subseteq I(\Lambda) \cdot Q$ by applying $- \otimes _\Lambda k$.

A complete projective resolution of $k$ in the case where $k$ has no square root of $-1$ is given by letting every term after the zeroth be $\Lambda \oplus \Lambda$ with differentials alternating between 
\[
\left( \begin{array}{cc} y & x \\ 1-x & y \end{array} \right) \,\,\, \mathrm{and} \,\,\, \left( \begin{array}{cc} -y & 1-x \\ x & y \end{array} \right).
\]
\end{ex}

\begin{subsection}{Bimodule resolutions of products} \label{bimod_res_subsection}

Let $\Lambda$ and $\Gamma$ be augmented algebras such that there exist bimodule resolutions $(P_*, d^P)$ and $(Q_*,d^Q)$  of $\Lambda$ and $\Gamma$ with $\im d^P \subset I(\Lambda) \cdot P_*+P_* \cdot I(\Lambda)$ and $\im d^Q \subset I(\Gamma) \cdot Q_*+Q_*\cdot I(\Gamma)$.  Furthermore assume $P_0=\Lambda ^e$ and $Q_0 = \Gamma ^e$.  Our goal in this section is to produce a bimodule resolution $(P \sqcup Q, \delta)$ for $\Lambda * \Gamma$.

We use a bar to denote an induction functor $\uparrow ^{(\Lambda * \Gamma)^e}$, and will write $d$ instead of $d^P$ or $d^Q$ --- which subalgebra we induce from, and which differential we are applying, should be clear by the context. Given bimodules $A$ and $B$ for an augmented algebra $\Delta$ we write $A \hat{\otimes} B$ for $A \otimes_\Delta k \otimes _\Delta B$.  If $A$ and $B$ are projective then so is $A \hat{\otimes} B$; this follows because $\Delta^e \hat \otimes \Delta^e \cong \Delta^e$.

We define $(P \sqcup Q, \delta)$  as follows: let $(P\sqcup Q)_0 = (\Lambda * \Gamma )^e $ and
\[
(P \sqcup Q)_n  = \bigoplus \cdots \hat{\otimes} \bar{P}_{i_{r-1}} \hat{\otimes} \bar{Q}_{i_r} \hat{\otimes} \bar{P}_{i_{r+1}} \hat{\otimes} \cdots 
\]
for $n>0$, where the direct sum is over all tuples $(i_1,i_2, \ldots)$ of strictly positive integers such that $\sum_j i_j = n$.  

To define the differential $\delta$, note that $(P \sqcup Q)_n$ is spanned by elements of the form
\[ a = a_1 \hat\otimes a_2 \hat\otimes \cdots \hat\otimes a_N \]
where the $a_i$ alternate between elements of $\bar{P}_*$ and $\bar{Q}_*$, and the sum of the degrees of the $a_i$ is $n$ (we say $a_i$ has degree $m$ if $a_i \in \bar{P}_m \cup \bar{Q}_m$).  Roughly, $\delta(a)$ will consist of two terms, one of which uses $\bar d^P$ or $\bar d^Q$ to drop the degree of $a_1$ by one (and leaves everything else alone), the other using doing the same to $a_N$.

If $N=1$ we put $\delta a_1 = \bar{d}a_1$.  This means that the positive degree parts of $(\bar{P}_*, \bar{d})$ and $(\bar{Q}_*, \bar{d})$ appear as subcomplexes of $(P\sqcup Q, \delta)$. If $N>1$ and $a_1$ and $a_N$ have degree greater than one we define
\begin{equation} \label{diff} 
\delta a = \bar{d}(a_1) \hat\otimes a_2 \hat\otimes \cdots \hat\otimes a_N + (-1)^n a_1 \hat\otimes a_2 \hat\otimes \cdots \hat\otimes \bar{d}(a_N) \end{equation}
so that if $r,s>1$, $\delta$ maps 
\[ \bar{P}_r \hat\otimes\cdots\hat\otimes\bar{Q}_s \to \bar{P}_{r-1}\hat\otimes\cdots\hat\otimes\bar{Q}_s \oplus\bar{P}_r\hat\otimes\cdots\hat\otimes\bar{Q}_{s-1} .\]
If $a_1$ has degree one then we modify the first term of (\ref{diff}) as follows.  If $a_1 \in \bar{P}_1$, we want this first term to be an element of $\bar{Q}_m \hat \otimes \cdots$ for some $m$.  Let
\begin{equation} \label{element_a} 
a= (x \otimes _\Lambda p \otimes_\Lambda y) \otimes _{\Lambda * \Gamma} \alpha \otimes_{\Lambda * \Gamma} (z \otimes _\Gamma q \otimes_\Gamma w) \otimes_{\Lambda * \Gamma} \cdots 
\end{equation}
be in $\bar{P}_1 \hat\otimes \bar{Q}_m \hat \otimes \cdots$ where $x,y,z,w \in \Lambda * \Gamma$, $\alpha \in k$, $p \in P_1$, $q \in Q_m$.  Then we define the $\bar{Q}_m \hat \otimes \cdots$ component of $\delta (a)$ to be
\[
\alpha x ( (1 \otimes \eps) \circ d)(p) \eps(y)\eps(z) \otimes_\Gamma q \otimes_\Gamma w \otimes_{\Lambda * \Gamma} \cdots 
\]
regarding $(1 \otimes \eps) \circ d$ as a map of left modules $P_1 \to \Lambda$ with image contained in $I(\Lambda)$.  Issues of well-definedness arise because of the $\otimes_\Lambda$, $\otimes_{\Lambda * \Gamma}$ and $\otimes_\Gamma$ in (\ref{element_a}).  The first two present no problems, the last works because if $\lambda \in I(\Lambda)$ and $\gamma \in \Gamma$ then $\lambda \gamma = \lambda \eps(\gamma)$ in $\Lambda * \Gamma$.  We make a similar modification to the first term of (\ref{diff}) when $a_1 \in \bar{Q}_1$, and to the second term in the case that the degree of $a_N$ is one.

This completes the definition of $(P\sqcup Q, \delta)$.  It is easily verified that $\delta^2 = 0$, and we now show the resulting complex is exact. Choose contracting homotopies $s,t$ for $(P_*, d)$ and $(Q_*, d)$ which are homomorphisms of left $\Lambda$ and $\Gamma$ modules respectively.  Then $s$ consists of a family of left $\Lambda$ module maps $s_i: P_i \to P_{i+1}$ for $i \geq 0$ and $s_{-1} : \Lambda \to P_0 =\Lambda^e$ such that 
\[
s_rd_{r+1}+d_{r+2}s_{r+1}=\id _{P_{r+1}}  \;\;\; r \geq -1 
\]
and $d_0 s_{-1}=\id_\Lambda$.  
We insist that  $s_{-1}(\lambda)=\lambda\otimes 1$ and $t_{-1}(\gamma)=\gamma\otimes 1$.  We use $s$ and $t$ to build a contracting homotopy $\sigma$ for $P \sqcup Q$ which is a homomorphism of left $\Lambda * \Gamma$ modules.  Firstly, define $\sigma_{-1}(x)=x \otimes 1$ for $x \in \Lambda * \Gamma$.  Now suppose that
\[
a=a_1 \hat\otimes \cdots \hat\otimes a_{N-1} \hat \otimes (1 \otimes_\Lambda p \otimes_\Lambda (\alpha + \gamma) )\in (\cdots \hat\otimes\bar{P}_m) \subset (P\sqcup Q)_n
\]
%
where $p \in P_m$, $\alpha \in k$ and $\gamma \in I(\Gamma)$.  Noting $s(1 \otimes \gamma) \in Q_1$ and $s(p) \in P_{m+1}$, define 
\begin{eqnarray*}
\sigma(a) &=& (-1)^{n+1} a_1 \hat\otimes \cdots \hat\otimes a_{N-1} \hat\otimes (1 \otimes_\Lambda s(p) \otimes_\Lambda \alpha) \\ & & \mbox{} + (-1)^{n+1} a_1 \hat\otimes \cdots  \hat\otimes (1 \otimes_\Lambda p \otimes_\Lambda 1) \hat\otimes (1\otimes_\Gamma t(1\otimes \gamma) \otimes_\Gamma 1).
\end{eqnarray*}
in $(\cdots \hat\otimes\bar{P}_{m+1}) \oplus (\cdots \hat\otimes\bar{P}_m \hat\otimes\bar Q _1)$.  A similar definition is made on summands of $(P \sqcup Q)_n$ ending in a term from $\bar Q$.  The reader may verify that $\sigma$ is a contracting homotopy for $(P\sqcup Q,\delta)$.
\end{subsection}

\begin{subsection}{Ordinary cohomology of products} \label{ord_cohom_prod}
We now use this bimodule resolution to study the ordinary cohomology of a product $\Lambda * \Gamma$ with the aim of proving the following theorem. 
\begin{theo} \label{main_theo}
Let $\Lambda$ and $\Gamma$ be augmented $k$-algebras which are either finite-dimensional or graded connected and finitely generated.  Then 
\[
\ext_{\Lambda * \Gamma} ^* (k,k) \cong \ext_\Lambda ^* (k,k) \sqcup \ext _\Gamma ^* (k,k) .
\]
\end{theo}
Define 
\[
(R_*, \partial) = ( (P \sqcup Q)_*\otimes_{\Lambda * \Gamma} k, \delta \otimes_{\Lambda * \Gamma}k).
\]
Then $H^n (R, \partial)$ computes $\tor _n ^{\Lambda * \Gamma}(\Lambda * \Gamma, k)$ which is zero for $n>0$ as $\Lambda  * \Gamma$ is projective, hence flat, as a right module over itself.  It follows that $(R, \partial)$ is a projective resolution of $k$ as a left $\Lambda * \Gamma$ module.  Because 
\[
\im \delta \subseteq I(\Lambda * \Gamma) \cdot P \sqcup Q + P \sqcup Q \cdot I(\Lambda * \Gamma)
\]
we have $\im \partial \subseteq I(\Lambda * \Gamma) \cdot R$, and so 
\[ 
\ext^n _{\Lambda * \Gamma} (k,k) \cong \Hom_{\Lambda * \Gamma} ( R_n, k). 
\]
Similarly, $(\mathsf{P}_*, \mathfrak{d}^P) = (P_*\otimes_\Lambda k, d^P \otimes_\Lambda k)$ and $(\mathsf{Q}_*, \mathfrak{d}^Q) = (Q_*\otimes_\Lambda k, d^Q \otimes_\Lambda k)$ are projective resolutions of $k$ such that 
\begin{eqnarray*} 
\ext^n_\Lambda (k,k) & \cong & \Hom_\Lambda (\mathsf{P}_n,k) \\
 & \cong & \Hom_\Lambda (P_n\otimes_\Lambda k, k) \\
& \cong & \Hom_k (k \otimes _\Lambda P_n\otimes_\Lambda k, k)
\end{eqnarray*}
with a similar result for the cohomology of $\Gamma$. 

\begin{lem} \label{vs_iso}
$\ext^n _{\Lambda * \Gamma} (k,k)$ is isomorphic as a $k$-vector space to the degree $n$ part of $E(\Lambda)\sqcup E(\Gamma)$.
\end{lem}
\begin{proof}
Our restrictions on $\Lambda$ and $\Gamma$ are enough to guarantee that $P_n$ is finitely-generated, so $\Hom_k (k \otimes _\Lambda P_n\otimes_\Lambda k, k) \cong k \otimes _\Lambda P_n \otimes _\Lambda k$.  
%
%
Writing $D$ for the duality functor $\Hom_k ( -, k)$ on the category of $k$-vector spaces,  a typical summand of the dual of $n$th degree part of $E(\Lambda) \sqcup E(\Gamma)$ looks like
\begin{align*} & \cdots \otimes_k D E^{i_{r-1}}(\Lambda) \otimes_k DE^{i_r} (\Gamma) \otimes_k DE^{i_{r+1}}(\Lambda) \otimes_k \cdots   \\
 \cong  & \cdots \otimes_\Gamma k \otimes_\Lambda P_{i_{r-1}} \otimes_\Lambda k \otimes_\Gamma Q_{i_r} \otimes_\Gamma k \otimes _\Lambda P_{i_{r+1}} \otimes_\Lambda k \otimes_\Gamma \cdots
\end{align*}
The $k$-dual of $E^n(\Lambda *\Gamma)$ is $k \otimes_{\Lambda * \Gamma} (P \sqcup Q)_n \otimes_{\Lambda * \Gamma} k=k\otimes_{\Lambda * \Gamma} R$, a typical summand of which is
%
%
%
\begin{align} \label{iso}
& \cdots \otimes_{\Lambda * \Gamma} \bar P_{i_{r-1}}  \otimes_{\Lambda * \Gamma} k \otimes_{\Lambda * \Gamma} \bar Q_{i_r}  \otimes_{\Lambda * \Gamma} k \otimes_{\Lambda * \Gamma} \bar P_{i_{r+1}}  \otimes_{\Lambda * \Gamma} k \otimes_{\Lambda * \Gamma} \cdots
 \\ \nonumber
\cong& \cdots  \otimes_\Gamma k \otimes_\Lambda P_{i_{r-1}} \otimes_\Lambda k \otimes_\Gamma Q_{i_r} \otimes_\Gamma k \otimes _\Lambda P_{i_{r+1}} \otimes_\Lambda k \otimes_\Gamma \cdots
\end{align}
The result follows.
\end{proof}


We now look at how $E(p_\Lambda)$ and $E(p_\Gamma)$ behave on the level of chain maps.  Regard $\ext^n_\Lambda(k,k)$ as the set of degree $n$ chain maps $\mathsf{P}_* \to \mathsf{P}_*$ modulo homotopy, so that multiplication in $E(\Lambda)$ corresponds to composition of chain maps.  Then $E(p_\Lambda)$ can be described as follows. Given an element $l$ of $\ext^n _\Lambda(k,k)$ pick a representative chain map $f_* : \mathsf{P}_* \to \mathsf{P}_*$, consisting of a family of maps $f_i : \mathsf{P}_{i+n} \to \mathsf{P}_i$ for $i \geq 0$.  Then $E(p_\Lambda)(l)$ can be represented by a chain map $F_*$ on $R_*$ which looks roughly like $\id \otimes f_*$.  Specifically, $F_*$ acts as follows.  It kills any summand of $R_{i+n}$ not of the form $\cdots \hat \otimes \bar P _{m}\otimes_{\Lambda *\Gamma} k$ for $m\geq n$. Now identify a summand of $R_{i+n}$ ending $\cdots  \hat\otimes \bar{Q}_r \hat\otimes \bar{P}_m \otimes _{\Lambda * \Gamma} k$  with
\[
\cdots \otimes_\Lambda k \otimes_\Gamma Q_r \otimes_\Gamma k \otimes _\Lambda P_m \otimes_\Lambda k
\]
similarly to (\ref{iso}). If $m>n$ then $F_i$ maps 
\[
\cdots \otimes _\Lambda 1 \otimes_\Gamma q \otimes _\Gamma 1 \otimes_\Lambda p \otimes_\Lambda 1 \,\,\, \mapsto \,\,\, \cdots \otimes _\Lambda 1 \otimes_\Gamma q \otimes_\Gamma 1 \otimes f(p \otimes_\Lambda 1)
\]
where the latter is in $\cdots \otimes_\Lambda k \otimes_\Gamma Q_r \otimes_\Gamma k \otimes _\Lambda P_{m-n} \otimes_\Lambda k\subset R_i$, and if $m=n$ then we think of $(\eps \otimes 1) \circ f_0$ as a map $P_m \otimes_\Lambda k \to k$, and let $F_i$ act by
\[
\cdots \otimes _\Lambda 1 \otimes_\Gamma q \otimes _\Gamma 1 \otimes_\Lambda p \otimes_\Lambda 1 \,\,\, \mapsto \,\,\, \cdots \otimes _\Lambda 1 \otimes_\Gamma q \otimes_\Gamma  ((\eps \otimes 1) \circ f_0)(p \otimes_\Lambda 1)
\]
in $\cdots \otimes_\Lambda k \otimes_\Gamma Q_r \otimes _\Gamma k \subset R_i$.
%
%

We now complete the proof of Theorem \ref{main_theo}.  Using Lemma \ref{vs_iso}, it is enough to prove that the morphism of augmented algebras $E(p_\Lambda) \sqcup E(p_\Gamma): E(\Lambda) \sqcup E(\Gamma) \to E(\Lambda * \Gamma)$ arising from $E(p_\Lambda)$ and $E(p_\Gamma)$ is onto.

Pick a word  $l_1 g_1 l_2 g_2 \cdots$ in $E(\Lambda) \sqcup E(\Gamma)$, where $l_i \in E(\Lambda), g_i \in E(\Gamma)$, let $l_i$ be represented by a chain map $f^i_*$ and $g_i$ by a chain map $f'^i_*$, and let $F^i_*$ and $F'^i_*$ be the chain maps representing $E(p_\Lambda)(l_i)$ and $E(p_\Gamma)(g_i)$ manufactured as above. Then $E(p_\Lambda) \sqcup E(p_\Gamma)$ sends $l_1 g_1 l_2 g_2 \cdots$ to the element of $E(\Lambda * \Gamma)$ represented by the composition $F^1_* \circ F'^1_* \circ F^2 _* \circ F'^2_* \circ \cdots $.

(\ref{iso}) shows that an element of $E^n(\Lambda * \Gamma)$ can be regarded as a tensor product $\alpha_1 \otimes_k \beta_1 \otimes_k \alpha_2 \otimes \cdots$ where each $\alpha_i$ lies in $\Hom_k(k \otimes_\Lambda P_r \otimes_k, k)\cong E^r(\Lambda)$ and each $\beta_i$ lies in $\Hom_k(k \otimes_\Lambda Q_s \otimes_k, k)\cong E^s(\Gamma)$ for some $r,s$ depending on $i$.  Let $l_i$ be the element of $E^r(\Lambda)$ corresponding to $\alpha_i$, and $g_i$ be the element of $E^s(\Gamma)$ corresponding to $\beta_i$.  Lifting each $\alpha_i$ to a chain map $a_i$ on $\mathsf{P}$, and each $\beta_i$ to a chain map $b_i$ on $\mathsf{Q}$, then taking chain maps $A^i_*$ and $B^i_*$ representing $E(p_\Lambda)(l_i)$ and $E(p_\Gamma)(g_i)$, we have that $A^1_* \circ B^1_* \circ A_* ^2 \circ \cdots$ is a chain map that lifts $\alpha_1 \otimes_k \beta_1 \otimes_k \alpha_2 \otimes \cdots$.  We have realised $\alpha_1 \otimes_k \beta_1 \otimes_k \alpha_2 \otimes \cdots$ as an element of $\im E(p_\Lambda) \sqcup E(p_\Gamma)$, finishing the proof.
\end{subsection}

\begin{subsection}{Hochschild cohomology of products} \label{hoch_cohom_prod}
We prove two main results on the structure of the Hochschild cohomology of a product.  Firstly:

\begin{prop} \label{nilp_HH}
Let $\Lambda$ and $\Gamma$ be non-trivial augmented algebras which are either finite dimensional or finitely generated graded connected.  Suppose that not both $E(\Lambda)$ and $E(\Gamma)$ are isomorphic to $k[x]/x^2$ as ungraded rings.  Then $\phi_k: \HH(\Lambda * \Gamma) \to E(\Lambda * \Gamma)$ is zero in positive degrees, and if $I(\Lambda)^N= I(\Gamma)^N=0$ then every element of positive degree in $\HH(\Lambda * \Gamma)$ has $N$th power zero.
\end{prop}

Secondly we determine the structure of a certain quotient of $\HH(\Lambda * \Gamma)$ in terms of the cohomology of $\Lambda$ and $\Gamma$.  Let $\iHH(\Lambda)$ be the unital subalgebra generated by the ideal $\ker \phi_k: \HH(\Lambda) \to E(\Lambda)$.  This is naturally an augmented algebra, so we can form $\iHH(\Lambda) * \iHH(\Gamma)$.  Let $A(\Gamma)$ be the annihilating ideal $\{ \gamma \in \Gamma: \gamma \Gamma = \Gamma \gamma =0 \}$.

\begin{theo} \label{hoch_prod}
Let $\Lambda$ and $\Gamma$ be finite dimensional or finitely generated graded connected algebras.  Then there 
is an ideal $R$ of $\HH(\Lambda * \Gamma)$ such that
\[ \HH(\Lambda * \Gamma) / R \cong (\iHH(\Lambda) * \iHH(\Gamma)) \oplus (E(\Lambda)\otimes_k A(\Gamma)) \oplus (E(\Gamma)\otimes_k A(\Lambda)) \]
as algebras.  The product of any two elements of the copy of $E(\Lambda)\otimes_k A(\Gamma)$ is zero in this quotient, as is the product of any two elements of $E(\Gamma)\otimes_k A(\Lambda)$. 
\end{theo}

To prove the first of these theorems we need a result on the graded centre.

\begin{lem} \label{gr_centre} 
 Let $R$ and $S$ be finitely generated graded connected $k$-algebras.
Then the graded centre of $R\sqcup S$ is concentrated in degree $0$ 
unless both $R$ and $S$ are isomorphic as ungraded rings to $k[x]/x^2$.  \end{lem}

\begin{proof}
If we fix bases $\{ r_i: i \in J_1 \}$ of $I(R)$ and $\{ s_i : s \in J_2 \}$ of $I(S)$ then $R \sqcup S$ has a basis consisting of the identity together with all words of the form $r_{i_1} s_{i_2} r_{i_3} \ldots r_{i_n}$, $s_{i_1} r_{i_2} s_{i_3} \ldots s_{i_n}$, $r_{i_1} s_{i_2} r_{i_3} \ldots s_{i_n}$ and $s_{i_1} r_{i_2} s_{i_3} \ldots r_{i_n}$.  We refer to these words as being of types RR, SS, RS, and SR respectively.

Let $z$ be a homogeneous element of the graded centre of $R \sqcup S$.  No word of type RR may appear in $z$, for given non-zero $s \in I(S)$ the product $sz$ would be non-zero and spanned by words beginning with an element of $s$ and therefore could not be equal to $\pm zs$.  Similarly no words of type SS may appear in $z$.

If $r \in I(R)$ is non-zero then $rz$ is spanned by words of type RR and type RS, whereas $zr$ is spanned by words of type RR and type SR.  Therefore only the type RR words may be non-zero.  It follows that any type RS word in $z$ begins with an element of $A(R)$. By symmetry every word in $z$ of type RS begins with an element of $A(R)$ and ends with an element of $A(S)$, and every word of type SR in $z$ begins with an element of $A(S)$ and ends with an element of $A(R)$.

Again let $r \in I(R)$ be non-zero. $rz$ consists only of words beginning with $r$.  The only elements of $I(R)$ beginning elements of $zr$ are in $A(R)$, and so every $r \in I(R)$ belongs to $A(R)$.  It follows $I(R)^2=0$, and similarly $I(S)^2=0$.  Comparing $rz$ and $zr$ for different choices of $r$ we see that $I(R)$ must be one-dimensional, and similarly $I(S)$.
\end{proof}

If $R=k[x]/x^2$ and $S=k[y]/y^2$ then the graded centre of $R \sqcup S$ may be non-trivial, for example with $x$ and $y$ in degree $1$ the element $xy+yx$ lies in the graded centre.

Proposition \ref{nilp_HH} can now be deduced:  recall that the image of $\phi_k$ lies in the graded centre of $E(\Lambda * \Gamma) \cong E(\Lambda) \sqcup E(\Gamma)$. Under the hypotheses of this theorem Lemma \ref{gr_centre} shows that this graded centre is zero in positive degrees.  The last statement follows from \cite[Proposition 4.4]{SS}.

Note that if $\Lambda$ is finite dimensional and has only one simple module $k$ then it is not possible for $E(\Lambda)$ to be isomorphic to $k[x]/x^2$ for dimension reasons.  If $\Lambda$ is finitely generated graded connected and $E(\Lambda) \cong k[x]/x^2$ then the results of \cite[\S 7]{wall} imply $\Lambda \cong k[x]$.

When $E(\Lambda)\cong E(\Gamma) \cong k[x]/x^2$, the map $\phi_k$ can be non-zero in positive degrees.
\begin{ex} \label{kx_example}
Let $k$ have characteristic not two and $\Lambda = k[x]$ so that
\[\begin{array}{lll}
 E(\Lambda) = k[X]/X^2 &  & E(\Lambda * \Lambda) = k\langle X,Y \rangle / (X^2, Y^2) \\
\HH^0(\Lambda) \cong k[x] & & \HH^1(\Lambda) \cong k[x]
\end{array} \]
and $\HH^n(\Lambda)$ is zero in higher degrees. Write $k\{ S \}$ for the free graded commutative ring on a set of generators $S$ with specified degrees.  Then the Hochschild cohomology ring of $\Lambda * \Lambda \cong k[x,y]/xy$ is
\[
\frac{k\{x_0, y_0, x_{1,r},y_{1,r}, \xi_2, \xi_3 : r \geq 1  \} }
{x_0y_0, 
x_0 x_{1,r}=x_{1,r+1},
y_0 x_{1,r},
y_0 y_{1,r}=y_{1,r+1},
y_0 x_{1,r},
x_0 \xi_i,
y_0 \xi_i,
\xi_2 x_{1,1} = \xi_3 = \xi_2 y_{1,1} }
\]
where the first subscripts denote degrees. 
The map $\phi_k$ is the augmentation in degree zero, zero in odd degrees, and $\phi_k (\xi_2) = XY+YX$.
\end{ex}


We now turn to the proof of Theorem \ref{hoch_prod}.  
Note that the inclusion $\iota: I(\Lambda) \hookrightarrow \Lambda * \Gamma$ is a map of $\Lambda * \Gamma$ bimodules, and the canonical projection $p_\Lambda$ induces $p_\Lambda ^e : (\Lambda * \Gamma)^e \to \Lambda ^e$.  Functoriality of $\ext$ gives a map
\[ \ext ^n_{p_\Lambda ^e} (p_\Lambda, \iota): \ext^n_{\Lambda^e} (\Lambda, I(\Lambda)) \to \ext_{(\Lambda * \Gamma)^e} ^n (\Lambda * \Gamma, \Lambda * \Gamma). \]
We could try to define a map $\iHH(\Lambda)\to \HH(\Lambda*\Gamma)$ by lifting an element of $\ker \phi_k \subset \iHH(\Lambda)$ to an element of $\ext_{\Lambda^e}(\Lambda, I(\Lambda))$ using (\ref{ILles}), then mapping this to $\HH(\Lambda*\Gamma)$ with $\ext ^n_{p_\Lambda ^e} (p_\Lambda, \iota)$.  In order for this to be well defined we would require the latter map to kill elements of $\ext_{\Lambda^e}(\Lambda, I(\Lambda))$ in the image of the connecting homomorphism $\omega_\Lambda$.  This does not happen in general: it can be shown that 
\[
\ext_{p_\Lambda ^e}(p_\Lambda, \iota) \circ \omega_\Lambda = \omega_{\Lambda * \Gamma} \circ \ext_{p_\Lambda ^e}(p_\Lambda, \id_k) 
\]
where $\omega_{\Lambda * \Gamma}$ is the connecting homomorphism in the long exact sequence arising from applying $\Hom_{(\Lambda*\Gamma)^e}(\Lambda * \Gamma, -)$ to the exact sequence of $\Lambda * \Gamma$ bimodules
\[0 \to \Lambda * \Gamma \to \Lambda \oplus \Gamma \to k \to 0\]
where $\Lambda$ and $\Gamma$ are $\Lambda * \Gamma$ modules via $p_\Lambda$ and $p_\Gamma$, and the map $\Lambda * \Gamma \to \Lambda \oplus \Gamma$ is $1 \mapsto (1,1)$.

We therefore begin by constructing the ideal $R$ referred to in the statement of Theorem \ref{hoch_prod}.  Let $P \sqcup Q$ be a bimodule resolution for $\Lambda * \Gamma$ constructed as in \ref{bimod_res_subsection}.  Write $\bar{P} * \bar{Q}$ for the subcomplex of $P\sqcup Q$ consisting of 
$(\Lambda * \Gamma)^e$ in degree $0$ and $\bar{P}_n \oplus\bar{Q}_n$ in degree $n>0$, with the induced differentials.  In positive degrees this complex is equal to $\bar{P} \oplus \bar{Q}$. There is a short exact sequence
\[ 0 \to \bar{P} * \bar{Q} \stackrel j \to P \sqcup Q \stackrel \pi \to E\to 0 \]
where $E$ is the quotient complex.  Because $(\bar{P} * \bar{Q})_n$ is a summand of $(P \sqcup Q)_n$, the following sequence is also exact:
\[ 0 \to \Hom _{(\Lambda * \Gamma)^e} (\bar{P} * \bar{Q}, \Lambda * \Gamma)  \to \Hom _{(\Lambda * \Gamma)^e} (P \sqcup Q, \Lambda * \Gamma) \to \Hom _{(\Lambda * \Gamma)^e} (E,\Lambda * \Gamma) \to 0. \]
Applying $\Hom_{(\Lambda * \Gamma)^e}(-, \Lambda* \Gamma)$ gives a long exact sequence
\begin{eqnarray}  \label{les}
\lefteqn{\cdots \stackrel {\omega} \to H^n (\Hom _{(\Lambda * \Gamma)^e} (E, \Lambda * \Gamma)) \stackrel{\pi^*}\to \HH^n (\Lambda * \Gamma) } 
\\ \nonumber
& &  \stackrel{j^*}\to H^n (\Hom _{(\Lambda * \Gamma)^e} (\bar{P} * \bar{Q}, \Lambda * \Gamma))  \stackrel{\omega}\to H^{n+1}(\Hom _{(\Lambda * \Gamma)^e} (E, \Lambda * \Gamma)) \stackrel{\pi^*}\to \cdots \end{eqnarray}
where $\omega$ is the connecting homomorphism. We now define $R$ to be $\im \pi^*$.  Elements of $\im \pi^*$ can be represented by cocycles on $P \sqcup Q$ that kill $\bar P * \bar Q$.  As $\bar P * \bar Q$ is a subcomplex, such  cocycles can be lifted to  chain maps that kill $\bar P * \bar Q$.  Because Hochschild cohomology is graded commutative it follows that $R=\im \pi^*$ is an ideal of $\HH(\Lambda * \Gamma)$.

We now examine the connecting homomorphism $\omega$ more carefully.  Let $a$ be an element of $H^n ( \Hom _{(\Lambda * \Gamma)^e}( \bar{P} * \bar{Q}, \Lambda * \Gamma))$ represented by a cocycle $\alpha : \bar{P}_n \to \Lambda * \Gamma$.  Then $\omega (a)$  is represented by the  map $E_{n+1} \to \Lambda * \Gamma$ sending 
\begin{eqnarray} \label{conn_hom} f \hat{\otimes} p \in \bar{Q}_1 \hat{\otimes} \bar{P}_n  & \mapsto  & d^l(f) \alpha(p) \\
\nonumber p \hat{\otimes} f \in \bar{P}_n \hat{\otimes} \bar{Q}_1 & \mapsto & \alpha(p) d^r(f)
\end{eqnarray}
and killing all other summands (we have abused notation here: $f \hat{\otimes} p$ should be replaced by its image in $E_{n+1}$).  Here, $d^l$ is the map $(1 \otimes \epsilon ) \circ \bar{d}$ regarded as a map $\bar{Q}_1 \to \Lambda * \Gamma$, and $d^r$ is the same with $\epsilon \otimes 1$ instead of $1 \otimes \epsilon$.

%
%
%

\begin{lem} \label{ker_conn_hm}
 A cocycle $\alpha: \bar{P} \to \Lambda * \Gamma$ represents an element of the kernel of the connecting homomorphism if and only if it has image contained in $I(\Lambda) \oplus A(\Gamma)$.
\end{lem}

\begin{proof}
(\ref{conn_hom}) shows that the element represented by such a cocycle has its image under the connecting homomorphism represented by the zero map. Conversely, suppose $\alpha: \bar{P} \to \Lambda * \Gamma$ is a cocycle and that the map (\ref{conn_hom}) above is a coboundary.  Note that the image of this map is contained in $I(\Gamma)$. In $E$, the differentials of (the images of) $f \hat{\otimes} p$ and $p \hat{\otimes} f \in E_n$ are in $E_{n-1} \cdot I(\Lambda)$ and $I(\Lambda) \cdot E_{n-1}$ respectively, by our hypothesis on the differential on $P$.  It follows that any coboundary maps the images of $\bar{Q}_1 \hat{\otimes} \bar{P}_{n-1} $ and $\bar{P}_{n-1} \hat{\otimes} \bar{Q}_1$ in $E$ into  $I(\Lambda)$, hence (\ref{conn_hom}) cannot be a coboundary unless it is zero.
\end{proof}

Notice that the complex $\bar{P}$ is not exact, in fact its homology computes 
\[
\tor_* ^{\Lambda^e} ( (\Lambda * \Gamma)^e, \Lambda) = \tor_* ^{\Lambda^e}( \Lambda^e \oplus I(\Gamma) \otimes_k I(\Gamma), \Lambda) 
\]
where the action of $\Lambda ^e$ on $I(\Gamma) \otimes_k I(\Gamma)$ is trivial.  Since $\Lambda^e$ is flat as a module over itself this is isomorphic to $ I(\Gamma)\otimes_k I(\Gamma) \otimes_k \tor_* ^{\Lambda} (k,k)$ as a vector space, where we have used \cite[X, Theorem 2.1]{CE} to identify $\tor_* ^{\Lambda^e}(k,\Lambda)$ with $\tor_*^\Lambda(k,k)$.  
%
%
Furthermore 
\begin{eqnarray}
\label{HnPstarQ}
H^n(\Hom_{(\Lambda * \Gamma)^e }(\bar{P} , \Lambda * \Gamma) ) &=& H^n (\Hom_{ \Lambda ^ e } (P, \Lambda \oplus I(\Gamma) ) ) 
\\ \nonumber
&=& \ext ^n _{\Lambda^e} (\Lambda, \Lambda) \oplus \ext ^n _{\Lambda ^e} (\Lambda, I(\Gamma)) 
\\ \nonumber
&=& \HH^n(\Lambda) \oplus \ext_\Lambda ^n (k,k) \otimes _k I(\Gamma).
\end{eqnarray}
The first summand comes from homomorphisms with image contained in $\Lambda \subset \Lambda * \Gamma$, and the second from those with image contained in $I(\Gamma) \subset \Lambda * \Gamma$.  

We now have enough information for an additive decomposition of the Hochschild cohomology groups of $\HH(\Lambda * \Gamma)$.

\begin{prop} \label{additive_decomp}
$\HH^0(\Lambda * \Gamma) \cong Z(\Lambda) * Z(\Gamma)$, and for $n>0$,
\[ \HH^n (\Lambda * \Gamma) \cong \iHH^n(\Lambda) \oplus \iHH^n(\Gamma) \oplus  E^n(\Lambda)\otimes_k A(\Gamma) \oplus E^n(\Gamma) \otimes_k A(\Lambda) \oplus (\im \pi^*)_n. \]
\end{prop}
\begin{proof}
The statement about $\HH^0(\Lambda * \Gamma)$ is clear.  In higher degrees, the long exact sequence (\ref{les}) gives short exact sequences
\[
0 \to (\im \pi^ *)_n \to \HH^n(\Lambda * \Gamma) \to (\ker \omega)_n \to 0.
\]
Using Lemma \ref{ker_conn_hm} and (\ref{HnPstarQ}) we can write $(\ker \omega)_n$ as a direct sum of $\iHH^n(\Lambda) \oplus E^n(\Lambda)\otimes_k A(\Gamma)$ and the corresponding object for $\Gamma$.  This completes the proof.
\end{proof}

The isomorphism in this proposition is realised as follows: elements of $\iHH(\Lambda)$ correspond to cocycles mapping $\bar{P}_* \subset P \sqcup Q$ to $I(\Lambda)\subset \Lambda * \Gamma$ (and killing other summands).  Elements of $E^n(\Lambda)\otimes_k A(\Gamma)$ correspond to cocycles mapping $\bar{P}_* \subset P \sqcup Q$ to $A(\Gamma)\subset \Lambda * \Gamma$ (and killing other summands).  Elements of $\im \pi ^*$ are represented by cocycles on $P \sqcup Q$ that kill $\bar P * \bar Q$.

Returning to Example \ref{kx_example} where $\Lambda=k[x]$, we have $A(\Lambda)=0$ so only the $\iHH$ and $\im \pi ^*$ terms appear in $\HH(\Lambda * \Lambda)$.  The two copies of $\iHH (\Lambda)$ are generated by $x_0, x_{1,r}$ and $y_0, y_{1,r}$ for $r\geq 0$.  The image of $\pi^*$ is the ideal $(\xi_2, \xi_3)$.

The last proposition means that there is a $k$-linear isomorphism between the groups in Theorem \ref{hoch_prod}, so we only need to show that it can be made multiplicative.  A careful reading of the proof of \cite[Proposition 4.4]{SS} shows that it implies the following result.

\begin{prop}
Let $\Delta$ be an augmented algebra and $(P_*, d)$ be a bimodule resolution of $\Delta$ with the property that $\im d \subset I(\Delta) \cdot P + P \cdot I(\Delta)$.  Let $J \subset I(\Delta)$ be an ideal of $\Delta$.  If $\xi, \eta \in \HH(\Delta)$ can be represented by cocycles $P_* \to \Delta$ with images contained in $I(\Delta)$ and $J$ respectively, then $\xi \eta$ can be represented by a cocycle with image contained in $I(\Delta)\cdot J + J\cdot I(\Delta)$.
\end{prop}

It follows immediately that elements of $\HH(\Lambda * \Gamma)$ which are identified with elements of $E(\Lambda) \otimes _k A(\Gamma)$ under the isomorphism in Proposition \ref{additive_decomp} have product zero with any element of $\ker \phi_k: \HH(\Lambda * \Gamma) \to \Lambda * \Gamma$.  In particular, this verifies the last sentence of Theorem \ref{hoch_prod}.  The only thing left to prove is that the elements of $\HH(\Lambda * \Gamma)$ identified with elements of $\iHH(\Lambda)$ and $\iHH(\Gamma)$ in Proposition \ref{additive_decomp} multiply, modulo $R$, like $\iHH(\Lambda) * \iHH(\Gamma)$.  That elements of the copies of $I(\iHH(\Lambda))$ and $I(\iHH(\Gamma))$ have product zero is clear. Given an element of $\iHH(\Lambda )$ we now show how to associate an element of $\HH(\Lambda * \Gamma)$ in a manner that induces a multiplicative map $\iHH(\Lambda ) \to \HH(\Lambda * \Gamma) /R$ whose image modulo $R$ is the copy of $\iHH(\Lambda)$ in $\HH(\Lambda * \Gamma)$ identified in Proposition \ref{additive_decomp}.

Let $f: P_n \to \Lambda$ be a cocycle representing an element of $\iHH(\Lambda)$, so that the image of $f$ is contained in $I(\Lambda)$.  Lift $f$ to a chain map $f_*$ on $P_*$.  We now build a chain map $F_*: P \sqcup Q \to P \sqcup Q$ using $\bar{f}_*$ that lifts the cocycle $\hat f: P \sqcup Q \to \Lambda * \Gamma$ sending $x \otimes_\Lambda p \otimes_\Lambda y \in \bar{P}_n$ to $xf(p)y \in \Lambda * \Gamma$, where $x,y \in \Lambda * \Gamma$ and $p \in P_n$.  The map $F_*$ is defined as follows:
\begin{enumerate}
\item $F(p) = \bar{f}_m (p)$ if $p \in \bar{P}_m$ and $m \geq n$.
\item $F( p \hat \otimes \cdots \hat \otimes q) = \bar f_m(p) \hat \otimes \cdots \hat \otimes q$ for $p \hat \otimes \cdots \hat \otimes q \in \bar{P}_m \hat \otimes \cdots \hat \otimes \bar Q_r$ and $m>n$.
\item $F( q \hat \otimes \cdots \hat \otimes p) = q \hat \otimes \cdots \hat \otimes \bar f_m(p)$ for $q \hat \otimes \cdots \hat \otimes p \in \bar{Q}_r \hat \otimes \cdots \hat \otimes \bar P_m$ and $m>n$.
\item $F( p_1 \hat \otimes \cdots \hat \otimes p_2) = \bar f_m(p_1) \hat \otimes \cdots \hat \otimes p_2 + p_1 \hat \otimes \cdots \hat \otimes \bar f_m(p_2)$ for $p_1 \hat \otimes \cdots \hat \otimes p_2 \in \bar{P}_{m_1} \hat \otimes \cdots \hat \otimes \bar P_{m_2}$ and $m_1, m_2 >n$
\end{enumerate}
On a term beginning or ending with $\bar P_n$, we make the following modification to the formulas above.  If $p \hat \otimes q_1 \hat \otimes \cdots \hat \otimes q_2 \in \bar{P}_n \hat \otimes \bar{Q}_r \hat \otimes \cdots \hat \otimes \bar Q_s$ then we define its image under $F$ to be $( (1 \otimes \eps) \circ \bar{f}_0) (p) \cdot q_1 \hat \otimes \cdots \hat \otimes q_2$
where $(1 \otimes \eps) \circ \bar{f}_0$ is considered as a map $\bar{P}_m \to \Lambda * \Gamma$.  Similar changes are made in the other cases.
Finally, $F$ is zero on all other summands of $P \sqcup Q$.  The reader may verify that $F$ is a chain map lifting $\hat f$.
Write $c(f_*)$ for the chain map $F_*$ built from $f_*$.

\begin{lem}
$c$ induces a well defined multiplicative map $\iHH(\Lambda) \to \HH(\Lambda * \Gamma)/R$.
\end{lem}
\begin{proof}
For well-definedness, we must show that if $f$ is a coboundary then $c(f_*)$ is zero in $\HH(\Lambda * \Gamma)/R$.  Suppose $f = g \circ d : P_n \to \Lambda$ is a coboundary.  Then $\hat g$ and  $\hat f \circ \delta$ agree on $\bar P * \bar Q$, so their difference is a cocycle that is zero on $\bar P * \bar Q$ hence represents an element of $R$.  The map is multiplicative because if $f_*$ and $g_*$ are two chain maps on $P_*$ then $c(f_* \circ g_*) - c(f_*) \circ c(g_*)$ is zero on $\bar P * \bar Q$.  The result follows.
\end{proof}

This completes the proof of Theorem \ref{hoch_prod}.

\end{subsection}

\end{section}

\bibliography{local_cohom}
\bibliographystyle{amsalpha}

\end{document}